\documentclass[11pt]{article}
\usepackage[T1]{fontenc}
\usepackage{lmodern,amsmath,amssymb,amsthm,mathtools,booktabs}
\usepackage[margin=1in]{geometry}
\usepackage[hidelinks]{hyperref}
\ifdefined\pdfinfoomitdate\pdfinfoomitdate=1\fi
\ifdefined\pdfsuppressptexinfo\pdfsuppressptexinfo=15\fi
\ifdefined\pdftrailerid\pdftrailerid{}\fi
\hypersetup{pdftitle={Sharp FRS and singularity thresholds for off-diagonal Gram maps},
 pdfauthor={Ying Xie},pdfsubject={Algebraic geometry and commutative algebra},
 pdfkeywords={orthogonal frames, Gram maps, FRS morphisms, convolutions, rational singularities, log canonical thresholds},
 pdfcreator={},pdfproducer={},pdfstartview={}}
\numberwithin{equation}{section}
\usepackage{microtype}
\newtheorem{theorem}{Theorem}[section]
\newtheorem{proposition}[theorem]{Proposition}
\newtheorem{lemma}[theorem]{Lemma}
\newtheorem{corollary}[theorem]{Corollary}
\theoremstyle{remark}
\newtheorem{remark}[theorem]{Remark}
\newcommand{\A}{\mathbb A}
\newcommand{\C}{\mathbb C}
\newcommand{\F}{\mathbb F}
\newcommand{\Z}{\mathbb Z}
\newcommand{\Sym}{\operatorname{Sym}}
\newcommand{\Hol}{\operatorname{Hol}}
\newcommand{\lct}{\operatorname{lct}}
\newcommand{\Spec}{\operatorname{Spec}}

\newcommand{\Gr}{\operatorname{OGr}}
\newcommand{\Dprime}{D_{\mathrm{prime}}}
\newcommand{\Drat}{D_{\mathrm{rat}}}
\title{Sharp FRS and singularity thresholds for off-diagonal Gram maps}
\author{Ying Xie\\Kennesaw State University\\
  \href{mailto:yxie2@kennesaw.edu}{\texttt{yxie2@kennesaw.edu}}}
\date{}
\begin{document}
\maketitle
\begin{abstract}
The off-diagonal Gram map for $h$ vectors in an $M$-dimensional quadratic
space is flat with reduced fibers of rational singularities (FRS) in
characteristic zero if and only if
$M>\max_{2\le r\le h}(2h-r+1-2h/r)$.
Equivalently, this gives the exact FRS convolution threshold of
$x\mapsto(x_ix_j)_{i<j}$. For its zero-fiber ideal $I_{M,h}$, we also prove
\[
\lct(I_{M,h})=\min\left\{\binom h2,\
 \min_{0\le k\le h-2}\left(\binom{k+1}{2}+\frac{M(h-k)}2\right)\right\}.
\]
In the strict range, all truncated coefficient maps have geometrically
integral complete-intersection fibers and are faithfully flat over
$\Z[1/2]$, uniformly over graphs on $h$ vertices. The proof uses a sharp
weighted estimate for hollow symmetric matrices over finite local
principal ideal rings, paired with isotropic lower bounds. Within the
established point-counting approach to rational singularities, this
estimate also gives an explicit modulus of continuity for Gram-map
densities over every local field of odd residue characteristic.
\end{abstract}

\noindent\textit{Preprint.} This manuscript has not undergone peer review.
Comments and corrections are welcome.

\section{Introduction}
Pairwise orthogonality imposes a structured family of quadratic equations
on an ordered collection of vectors. These equations define orthogonal
frame schemes; choosing only selected pairs gives the
Lov\'asz--Saks--Schrijver ideals associated with graphs
\cite{LSS,CW}. Even when the
full frame scheme is a normal complete intersection, its singularities
need not be rational. The problem is to determine precisely how large
the ambient quadratic space must be to obtain rational singularities,
and to describe the singularities below that threshold.

We determine the exact FRS threshold for every number of vectors
using quadratic Fourier sums over finite local principal ideal rings.
The estimates include all Smith valuation patterns and are uniform
in the truncation length and the prescribed Gram entries. Besides the rationality
classification, the argument determines the log canonical threshold of the Gram
ideal even when the zero scheme is reducible or nonreduced. It also gives
integral coefficient fibers for arbitrary prescribed inner products and
for arbitrary subgraphs in a sharp range uniform over graphs.

\subsection{Notation and main results}
Fix integers $h\ge2$ and $M\ge1$, and put $e=\binom h2$ and $N=Mh$.
For an $M\times h$ matrix $V=(v_{ai})$, let
\[
 \Phi_{M,h}(V)=\left(\sum_{a=1}^M v_{ai}v_{aj}\right)_{i<j}.
\]
Write $I_{M,h}$ for the ideal generated by these $e$ polynomials and
$X_{M,h}=\Spec\C[V]/I_{M,h}$. Thus the scheme structure, and not only
its reduced support, is part of the notation.
We use the standard nondegenerate symmetric bilinear form. Over an
algebraically closed field of characteristic different from two, every
nondegenerate symmetric form is equivalent to this one. Singularity
invariants are taken over $\C$; the corresponding characteristic-zero
statements are geometric statements after extension to an algebraic
closure. The scheme $(X)_{L-1}$ denotes the jet scheme representing maps
$\operatorname{Spec} A[t]/(t^L)\to X$.

For an ideal $I$ on a smooth variety $Y$, we use the convention
\[
 \lct_Y(I)=\inf_E\frac{A_Y(E)}{\operatorname{ord}_E(I)},
\]
where $A_Y(E)$ is the log discrepancy and the infimum is over prime
divisors over $Y$ with positive order on $I$. No reducedness convention
is imposed on the ideal. We write $\lct(I)$ when the ambient affine
space is understood.
Define
\begin{align}
 \delta_h(r)&=\frac{r(2h-r+1)}2-h,\qquad 2\le r\le h, \label{eq:delta}\\
 F(h)&=\max_{2\le r\le h}\frac{2\delta_h(r)}r
      =\max_{2\le r\le h}\left(2h-r+1-\frac{2h}{r}\right),\label{eq:F}\\
 \Drat(h)&=\lfloor F(h)\rfloor+1.\label{eq:Drat}
\end{align}

Following Aizenbud--Avni \cite{AAfrs}, a morphism between smooth
characteristic-zero varieties is \emph{FRS} if it is flat and every
geometric fiber is reduced and has rational singularities.

\begin{theorem}\label{thm:rat}
The following are equivalent:
\begin{enumerate}
\item $M>F(h)$, or equivalently $M\ge\Drat(h)$;
\item the off-diagonal Gram map $\Phi_{M,h}$ is FRS;
\item its zero fiber $X_{M,h}$ has rational singularities.
\end{enumerate}
In this range every fiber is an integral complete intersection with
rational singularities. The same sufficient condition holds for every
 graph-Gram map on $h$ vertices, obtained by retaining any subset of
 the output coordinates. It is sharp as a condition uniform over all
 such graphs.
\end{theorem}

For $\phi_h:\A^h\to\A^e$, $\phi_h(x)=(x_ix_j)_{i<j}$, the sum
$\phi_h(x^{(1)})+\cdots+\phi_h(x^{(M)})$ is precisely $\Phi_{M,h}$.
Thus Theorem~\ref{thm:rat} determines the least FRS convolution power
of $\phi_h$: it is $\Drat(h)$. The graph statement concerns a condition
valid for \emph{all} graphs on $h$ vertices, not a necessary threshold
for each individual graph.

\begin{theorem}\label{thm:lct}
For all $M\ge1$ and $h\ge2$, the log canonical threshold in $\A^N_\C$ is
\begin{equation}\label{eq:lct}
 \lct(I_{M,h})
 =\min\left\{e,\ \min_{0\le k\le h-2}
       \left(\binom{k+1}{2}+\frac{M(h-k)}2\right)\right\}.
\end{equation}
The global threshold equals the threshold at the origin.
\end{theorem}

For $L\ge1$, expand
$V(t)=\sum_{a=0}^{L-1}V_a t^a$ modulo $t^L$ and let $\Phi^{[L]}_{M,G}$
record all $L$ coefficients of each inner product indexed by an edge
of a simple graph $G$ on $h$ vertices.

\begin{theorem}\label{thm:coeff}
If $M>F(h)$, then for every $G$ and $L$ the morphism
\[
 \Phi^{[L]}_{M,G}:\A_{\Z[1/2]}^{LMh}
       \longrightarrow\A_{\Z[1/2]}^{L|E(G)|}
\]
is faithfully flat with geometrically integral complete-intersection
fibers of dimension $L(Mh-|E(G)|)$.
\end{theorem}

\subsection{Relation to earlier work}
The relation between finite-ring point counts and rational singularities
is an established framework. Aizenbud--Avni \cite{AAcount} connect these
counts with rational singularities and representation growth; Glazer
\cite{Glazer} completes the criterion for local complete intersections.
Cluckers--Glazer--Hendel \cite{CGH} give a relative characterization
uniform over the target, and Aizenbud--Avni \cite{AAfrs} relate FRS maps
to continuous densities of pushforward measures over local fields.
Glazer--Hendel \cite[Theorem 1.7]{GH} prove eventual FRS regularity under
convolution when the image spans the target affine space. This applies
to $\phi_h$, since $\phi_h(0)=0$ and $\phi_h(e_i+e_j)$ runs through the
standard basis of $\A^e$. Our convolution identity places the Gram
family in this context and gives an exact exponent for this map.
The general point-counting-to-singularity principle is not new here.

The specific input is Lemma~\ref{lem:hollow}, which controls the full
weighted sum over hollow Smith strata with the sharp exponential rate.
Its uniformity in the ring and the conductor, together with matching
isotropic lower bounds, yields the threshold classification. It also
gives quantitative density estimates at every odd residue characteristic
(Proposition~\ref{prop:density}), without discarding small odd primes.
Our direct passage through Lang--Weil, constructibility, and jets records
the additional universal coefficient-fiber statement over $\Z[1/2]$.
Related uses of jet dimensions, rational singularities, and arithmetic
counts occur for quiver moment maps in work of Budur \cite{Budur} and
Vernet \cite{Vernet}.

The graph ideals originate in orthogonal representations studied by
Lov\'asz--Saks--Schrijver \cite{LSS}; Conca--Welker \cite{CW} develop their
algebraic relation to coordinate sections of determinantal varieties.
For generic determinantal ideals, Johnson \cite{Johnson} computes
multiplier ideals and log canonical thresholds, and Docampo
\cite[Theorem 5.6]{Docampo} obtains the threshold via arcs. These give
the relevant determinantal context, but their generic-minor formula
does not directly compute the threshold of the off-diagonal Gram ideal.

Casabella and Sammartano determine the complete-intersection and prime
thresholds for orthogonal frames \cite[Theorems 1--2]{CS}.
In particular, their prime threshold is
\begin{equation}\label{eq:Dprime}
\Dprime(h)=\min\left\{
2\left\lfloor\frac{2h+1-\sqrt{8h+1}}2\right\rfloor+2,\
2\left\lfloor\frac{2h+1-\sqrt{8h-7}}2\right\rfloor+1
\right\},
\end{equation}
and $M\ge\Dprime(h)$ gives a normal complete intersection.
They ask whether this range already gives ($F$-)rational singularities
\cite[Question 8.1]{CS}. Theorems \ref{thm:lct} and \ref{thm:rat}
give the sharp replacement in characteristic zero. The nonrational
examples at the exceptional prime thresholds were established in
\cite[Theorem 1.3 and Sections 6--7]{Xie} using generic transverse
singularities. Section~\ref{sec:sharpness} gives a first-jet proof of
the same obstruction. The earlier paper also studies defining ideals and
depth of principal components, and proves non-$F$-rationality for its
exceptional family in every odd characteristic. Those results are
distinct from the positive classification proved here. The positive
estimate closes the gap left
by these counterexamples and also yields the exact ideal threshold
and the coefficient-fiber theorem. The rationality criterion is
proved independently of the prime-threshold theorem. We use
\cite{CS} for the comparison with the normal range, including
Table~\ref{tab:thresholds} and the normal-range corollaries. The positive
argument does not use the principal-component results of \cite{Xie}.

For complete graphs with $h\ge2$, the sufficient condition $M\ge3h-4$
follows from \cite[arXiv version, Lemma 2.5 and Theorem 3.2]{TV}, since
$\operatorname{pmd}(K_h)\le2h-3$ and $k(K_h)=h-1$.
For sparse graphs, graph-specific bounds can be smaller than the
uniform bound proved here; sharpness here means sharpness over all
graphs on $h$ vertices, including the complete graph. The present
threshold has asymptotic form $2h-2\sqrt{2h}+O(1)$.
Our positive-characteristic conclusion is Theorem \ref{thm:coeff};
no claim of strong $F$-regularity or $F$-rationality in that full range is
made. The sharp positive-characteristic sufficiency question remains open.

\subsection{Boundary cases and examples}
The distinction between normality, log canonicity, and rationality is
particularly small but nontrivial. We prove that $\Drat(h)$ is always
$\Dprime(h)$ or $\Dprime(h)+1$. In the normal range,
$X_{M,h}$ is log canonical exactly when $M\ge F(h)$.
Consequently the normal log-canonical but nonrational cases are exactly
\[
 h=2a^2,\qquad M=(2a-1)^2,\qquad a\ge2.
\]
The first is $(M,h)=(9,8)$. In contrast, $(M,h)=(15,12)$ is normal
but not log canonical: its ideal threshold is $131/2$, whereas its
codimension is $66$. For $(M,h)=(3,4)$, well outside the normal range,
the ideal threshold is $11/2$.

Table~\ref{tab:thresholds} illustrates the distinction. Its log-canonical
column is the smallest $M$ in the normal range with log-canonical
singularities, namely $\max\{\Dprime(h),\lceil F(h)\rceil\}$.

\begin{table}[htbp]
\centering
\begin{tabular}{rrrr}
\toprule
$h$ & Normal & Log canonical & Rational \\
\midrule
2 & 2 & 2 & 2 \\
3 & 3 & 3 & 3 \\
4 & 4 & 4 & 4 \\
6 & 7 & 7 & 7 \\
7 & 8 & 8 & 8 \\
8 & 9 & 9 & 10 \\
11 & 14 & 14 & 14 \\
12 & 15 & 16 & 16 \\
16 & 22 & 22 & 22 \\
18 & 25 & 25 & 26 \\
\bottomrule
\end{tabular}
\caption{Selected sharp ambient-dimension thresholds for the full zero fiber.}
\label{tab:thresholds}
\end{table}

The exact threshold is also detected by the zeroth and first jet schemes.
Writing $c=\lct(I_{M,h})$, we show
\[
 c=N-\max\left\{\dim X_{M,h},\frac12\dim(X_{M,h})_1\right\}.
\]
Moreover, $\dim(X_{M,h})_{L-1}=L(N-c)$ whenever $M$ or $L$ is even.
In the normal range the first jet alone detects rationality and log
canonicity. These properties are consequences of the structure of this
family, rather than general properties of jet schemes.

\subsection{Proof strategy}
The proof has three parts. First, a unit diagonal pivot estimates
weighted counts of arbitrary symmetric matrices over finite local
rings. A two-by-two unit off-diagonal pivot reduces the hollow case to
that estimate; the residual diagonal constraints are scalar product
equations $bc=t$. Second, Fourier inversion gives counts of prescribed
Gram fibers. Lang--Weil and a separate smooth-point argument turn these
counts into integral complete-intersection coefficient fibers, and
constructibility passes from finite fields to the universal model over
$\Z[1/2]$. Third, jet-scheme criteria give rational singularities and the
lower bound for the ideal threshold. Maximal-isotropic strata provide
the opposite bound. For rationality necessity, an elementary
integrality obstruction first gives the expected complete-intersection
dimension; isotropic strata and their first jets then exclude the
entire complementary range.

The proofs of the uniform statements do not depend on computations.
The ancillary code provides exact finite controls for the arithmetic,
Fourier normalization, and first-jet counts.

\section{Weighted counts of symmetric matrices}
Let $\mathcal O$ be a discrete valuation ring with finite residue
field of odd cardinality $q$, uniformizer $\pi$, and
$R_\ell=\mathcal O/(\pi^\ell)$. Write $\nu_\ell(x)$ for the valuation
of $x\in R_\ell$, with $\nu_\ell(0)=\ell$. Thus
$|\pi^vR_\ell|=q^{\ell-v}$ for $0\le v\le\ell$.
For a square matrix $A$ over $R_\ell$, put
\[
 \rho_\ell(A)=\log_q|\operatorname{im} A|.
\]
If its Smith factors have valuations $a_i$, allowing $a_i=\ell$,
then $\rho_\ell(A)=\sum_i(\ell-a_i)$. Invertible row and column
operations preserve $\rho_\ell$. A matrix is called \emph{primitive}
if at least one entry is a unit; it need not be invertible.

\subsection{Common valuation and scalar products}
We first record the two counting facts used in the matrix estimates.

\begin{lemma}[Removal of the common valuation]\label{lem:content}
Let $\mathcal M_k(R)$ denote either $\Sym_k(R)$ or $\Hol_k(R)$,
the symmetric or hollow symmetric matrices, respectively. Every nonzero
$A\in\mathcal M_k(R_L)$ determines a unique pair $(\ell,A_0)$ with
\[
 A=\pi^{L-\ell}\widetilde A_0,\qquad
 1\le\ell\le L,\quad A_0\in\mathcal M_k(R_\ell)\text{ primitive},
\]
where the equality is independent of the lift $\widetilde A_0$ to
$R_L$. Here $L-\ell$ is the minimum valuation of the entries of $A$,
and
\begin{equation}\label{eq:contentlength}
 \rho_L(A)=\rho_\ell(A_0).
\end{equation}
Consequently, for every $\sigma >0$,
\begin{equation}\label{eq:contentmass}
 \sum_{A\in\mathcal M_k(R_L)}q^{-\sigma \rho_L(A)}
 =1+\sum_{\ell=1}^{L}
       \sum_{\substack{A_0\in\mathcal M_k(R_\ell)\\A_0\text{ primitive}}}
              q^{-\sigma \rho_\ell(A_0)}.
\end{equation}
\end{lemma}
\begin{proof}
Multiplication by $\pi^{L-\ell}$ induces an additive bijection
\[
 \iota_\ell:R_\ell\longrightarrow\pi^{L-\ell}R_L.
\]
Its injectivity follows because
$\pi^{L-\ell}x=0$ modulo $\pi^L$ exactly when $x=0$ modulo
$\pi^\ell$. Applied entrywise, this bijection identifies primitive
matrices over $R_\ell$ with matrices of common valuation $L-\ell$
over $R_L$. Symmetry and zero diagonal entries are preserved.
For a vector $x\in R_L^k$,
\[
 Ax=\iota_\ell(A_0\bar x),\qquad \bar x=x\bmod\pi^\ell.
\]
Reduction onto $R_\ell^k$ is surjective, so
$\operatorname{im} A=\iota_\ell(\operatorname{im} A_0)$ and
\eqref{eq:contentlength} follows. The valuation strata are disjoint
and exhaust the nonzero matrices, proving \eqref{eq:contentmass}.
There is no additional lift factor: different lifts of $A_0$ give
the same matrix $A$ after multiplication by $\pi^{L-\ell}$.
\end{proof}

\begin{lemma}[A scalar product count]\label{lem:products}
For $t\in R_\ell$, let
$N_\ell(t)=|\{(b,c)\in R_\ell^2:bc=t\}|$. Then
\begin{equation}\label{eq:productcount}
 N_\ell(t)=
 \begin{cases}
  (w+1)(q-1)q^{\ell-1},&\nu_\ell(t)=w<\ell,\\
  q^\ell+\ell(q-1)q^{\ell-1},&t=0.
 \end{cases}
\end{equation}
In particular, $N_\ell(t)\le(\ell+1)q^\ell$ for every $t$.
\end{lemma}
\begin{proof}
There are $(q-1)q^{\ell-v-1}$ elements $b$ of valuation $v<\ell$.
For each, multiplication by $b$ has image $\pi^vR_\ell$ and a
kernel of cardinality $q^v$. Thus $bc=t$ has exactly $q^v$
solutions in $c$ if $t\in\pi^vR_\ell$, and none otherwise.
Each eligible valuation $v$ contributes $(q-1)q^{\ell-1}$ pairs.
For nonzero $t$ the eligible valuations are $0,\ldots,w$.
For $t=0$ they are $0,\ldots,\ell-1$, and $b=0$ contributes
the additional $q^\ell$ pairs.
\end{proof}

\subsection{The symmetric-matrix recurrence}
For a real number $\sigma >0$, define
\[
 Z_k(\ell,\sigma )=\sum_{A\in\Sym_k(R_\ell)}q^{-\sigma \rho_\ell(A)},\qquad
 a_k(\sigma )=\max_{0\le j\le k}
          \left\{\frac{j(2k-j+1)}2-\sigma j\right\}.
\]
The zero-by-zero convention is $Z_0=1$ and $a_0=0$.

\begin{lemma}\label{lem:symmetric}
Set $C_0=1$ and $C_k=1+\binom{k+1}{2}C_{k-1}$ for $k\ge1$.
Then
\begin{equation}\label{eq:symmetricbound}
 Z_k(\ell,\sigma )\le C_k(\ell+1)^k q^{\ell a_k(\sigma )}
\end{equation}
for every odd $q$, every $\ell\ge1$, and every real $\sigma >0$.
\end{lemma}
\begin{proof}
Let $P_k(\ell,\sigma )$ be the same sum restricted to primitive matrices.
A primitive symmetric matrix has $v^tAv$ a unit for at least one of
the vectors
\[
 e_i\quad(1\le i\le k),\qquad e_i+e_j\quad(i<j).
\]
Indeed, if every diagonal entry is a nonunit, a unit off-diagonal
entry gives such a vector because $2$ is invertible. Each vector
extends to a fixed unimodular basis, whose congruence action is a
bijection preserving image length. These $\binom{k+1}{2}$ charts
cover the primitive matrices. They may overlap; summing their
nonnegative masses gives an upper bound.

In one chart write
\[
 A=\begin{pmatrix}a&b^t\\ b&C\end{pmatrix},\qquad a\in R_\ell^\times,
 \qquad S=C-bb^t/a.
\]
Invertible block elimination changes $A$ into $\operatorname{diag}(a,S)$,
so $\rho_\ell(A)=\ell+\rho_\ell(S)$. Conversely, each
$(a,b,S)\in R_\ell^\times\times R_\ell^{k-1}\times\Sym_{k-1}(R_\ell)$
gives exactly one matrix in the chart, by $C=S+bb^t/a$.
In particular $S$ is unrestricted, including $S=0$ and matrices
all of whose entries are nonunits. Since the number of pairs $(a,b)$
is $(q^\ell-q^{\ell-1})q^{(k-1)\ell}\le q^{k\ell}$,
\begin{equation}\label{eq:symunit}
 P_k(\ell,\sigma )\le\binom{k+1}{2}q^{(k-\sigma )\ell}Z_{k-1}(\ell,\sigma ).
\end{equation}
Lemma~\ref{lem:content} now gives the explicit recurrence
\begin{equation}\label{eq:symrecurrence}
 Z_k(L,\sigma )=1+\sum_{\ell=1}^{L}P_k(\ell,\sigma )
 \le1+\binom{k+1}{2}\sum_{\ell=1}^{L}
                       q^{(k-\sigma )\ell}Z_{k-1}(\ell,\sigma ).
\end{equation}

Induct on $k$. Separating $j=0$ and replacing $j$ by $j+1$ in
the other terms of the maximum gives
\[
 a_k(\sigma )=\max\{0,k-\sigma +a_{k-1}(\sigma )\}.
\]
For every $1\le\ell\le L$, including when
$k-\sigma +a_{k-1}(\sigma )<0$,
\[
 q^{\ell(k-\sigma +a_{k-1}(\sigma ))}\le q^{La_k(\sigma )}.
\]
Use this inequality and
$\sum_{\ell=1}^{L}(\ell+1)^{k-1}\le(L+1)^k$
in \eqref{eq:symrecurrence}. Since $a_k(\sigma )\ge0$, the initial
term $1$ is absorbed by the stated value of $C_k$.
\end{proof}

\subsection{An exact hollow pivot count}
Let $U_h(\ell,\sigma )$ be the weighted mass of primitive matrices in
$\Hol_h(R_\ell)$, for $h\ge2$.

\begin{lemma}\label{lem:hollow}
With $B_h=\binom h2 C_{h-2}$, one has
\begin{equation}\label{eq:hollow}
 U_h(\ell,\sigma )\le B_h(\ell+1)^{2h-4}q^{\ell b_h(\sigma )},\qquad
 b_h(\sigma )=\max_{2\le r\le h}\{\delta_h(r)-\sigma r\}.
\end{equation}
Moreover,
\begin{equation}\label{eq:fullhollow}
 \sum_{H\in\Hol_h(R_L)}q^{-\sigma \rho_L(H)}
     =1+\sum_{\ell=1}^L U_h(\ell,\sigma ).
\end{equation}
\end{lemma}
\begin{proof}
Every primitive hollow matrix has a unit off-diagonal entry.
Let $W_h(\ell,\sigma )$ be the weighted mass in the chart $H_{12}\in R_\ell^\times$.
Permuting indices identifies the other $\binom h2$ charts with this
one. Coverage and nonnegativity give
\begin{equation}\label{eq:hollowcover}
 U_h(\ell,\sigma )\le\binom h2 W_h(\ell,\sigma ).
\end{equation}
No disjointness of these charts is required.

Put $k=h-2$. In the fixed chart write
\[
 H=\begin{pmatrix}A&B\\ B^t&C\end{pmatrix},\qquad
 A=\begin{pmatrix}0&a\\a&0\end{pmatrix},\quad a\in R_\ell^\times,
 \qquad S=C-B^tA^{-1}B.
\]
The invertible matrix
\[
 E=\begin{pmatrix}I_2&-A^{-1}B\\0&I_k\end{pmatrix}
 \quad\text{satisfies}\quad
 E^tHE=\begin{pmatrix}A&0\\0&S\end{pmatrix}.
\]
As $A$ is invertible,
\begin{equation}\label{eq:hollowlength}
 \rho_\ell(H)=2\ell+\rho_\ell(S).
\end{equation}
This identity uses the full image length of $S$, irrespective of its
rank modulo $\pi$.

Fix $a\in R_\ell^\times$ and an arbitrary $S\in\Sym_k(R_\ell)$.
If column $j$ of $B$ is $(b_j,c_j)^t$, the condition that
$C=S+B^tA^{-1}B$ have zero diagonal is exactly
\[
 b_jc_j=-aS_{jj}/2\qquad(1\le j\le k).
\]
These equations use disjoint pairs of variables. Each choice of $B$
satisfying them determines a unique hollow $C$; its off-diagonal
entries impose no further conditions. Thus the fiber over $(a,S)$
has exactly $\prod_{j=1}^k N_\ell(-aS_{jj}/2)$ elements. Combining
this count with \eqref{eq:hollowlength} gives the exact chart identity
\begin{equation}\label{eq:hollowexact}
 W_h(\ell,\sigma )=q^{-2\sigma \ell}
   \sum_{a\in R_\ell^\times}\ 
   \sum_{S\in\Sym_k(R_\ell)}
      q^{-\sigma \rho_\ell(S)}\prod_{j=1}^k N_\ell(-aS_{jj}/2).
\end{equation}
In particular this sum includes every singular or nonprimitive
residual matrix $S$. The case $S_{jj}=0$ is covered by the second
line of \eqref{eq:productcount}.

Use $N_\ell(t)\le(\ell+1)q^\ell$ and
$|R_\ell^\times|\le q^\ell$ in \eqref{eq:hollowexact}, then apply
\eqref{eq:hollowcover}. This yields
\begin{equation}\label{eq:hollowpivotbound}
 U_h(\ell,\sigma )\le\binom h2(\ell+1)^k
              q^{(h-1-2\sigma )\ell}Z_k(\ell,\sigma ).
\end{equation}
Lemma~\ref{lem:symmetric} supplies a second factor $(\ell+1)^k$.
The exponent becomes
\[
 h-1-2\sigma +a_{h-2}(\sigma )
 =\max_{0\le j\le h-2}\{\delta_h(j+2)-\sigma (j+2)\}=b_h(\sigma ),
\]
proving \eqref{eq:hollow}. Finally, apply Lemma~\ref{lem:content}
to hollow matrices to obtain \eqref{eq:fullhollow}.
\end{proof}

\begin{remark}[Boundary case and mixed valuations]\label{rem:mixedvaluation}
For $h=2$ the residual matrix is empty, so the product in
\eqref{eq:hollowexact} equals one and
\[
 U_2(\ell,\sigma )=(q^\ell-q^{\ell-1})q^{-2\sigma \ell}.
\]
Thus $B_2=1$ and the length exponent $2h-4=0$ cause no exception.
For a nontrivial valuation example, over $\Z/27$ take
\[
 H=3\begin{pmatrix}0&1&1\\1&0&3\\1&3&0\end{pmatrix}.
\]
Its common valuation is one. Lemma~\ref{lem:content} replaces it by
the displayed primitive matrix $H_0$ over $\Z/9$, without changing
image length. The pivot $a=1$ leaves $S=(-6)=(3)$ over $\Z/9$.
Therefore
\[
 \rho_3(H)=\rho_2(H_0)=4+\rho_2((3))=5.
\]
The Smith valuations of $H$ are $(1,1,2)$. Its reduction modulo $3$
is zero, while the primitive reduction has rank two; neither residue
rank alone records the residual contribution to image length.
\end{remark}

For comparison, \eqref{eq:hollowexact} gives the following complete
count in the chart $H_{12}\in(\Z/9)^\times$ for $h=3$.
There are six choices of $a$; the three rows account for all
$6\cdot9^2=486$ matrices in this chart.
\begin{center}
\begin{tabular}{lrrrr}
\toprule
Residual $S$ & Choices of $S$ & $N_2(-aS/2)$ & Matrices $H$ & $\rho_2(H)$\\
\midrule
$0$ & 1 & 21 & 126 & 4\\
$3$ or $6$ & 2 & 12 & 144 & 5\\
A unit & 6 & 6 & 216 & 6\\
\bottomrule
\end{tabular}
\end{center}
The middle row records precisely the nonzero nonunit residuals.

\section{Fourier inversion and a uniform fiber estimate}
A primitive additive character $\psi:R_L\to\C^\times$ means one
whose restriction to the socle $\pi^{L-1}R_L$ is nontrivial.
The pairing $(a,b)\mapsto\psi(ab)$ is nondegenerate: for $a\ne0$,
the nonzero ideal $aR_L$ contains the socle. Such characters exist
for these finite local principal ideal rings.
For $H\in\Hol_h(R_L)$ set
\[
 G_L(H)=q^{-Lh}\sum_{x\in R_L^h}\psi(x^tHx/2).
\]

\begin{lemma}\label{lem:gauss}
For every symmetric $H$, including singular $H$,
\[
 |G_L(H)|^2=q^{-\rho_L(H)}.
\]
\end{lemma}
\begin{proof}
In the square of the absolute value change variables from $(x,y)$
to $u=x-y$ and $v=(x+y)/2$. This is a bijection, and the exponent
becomes $u^tHv$. Summing over $v$ gives $q^{Lh}$ if $Hu=0$ and
zero otherwise. Thus $|G_L(H)|^2=|\ker H|/q^{Lh}$.
\end{proof}

For $c\in R_L^e$, character orthogonality and the independence of
the $M$ rows of $V$ give
\begin{equation}\label{eq:Fourier}
 \frac{|\Phi_{M,h}^{-1}(c)(R_L)|}{q^{L(N-e)}}
   =\sum_{H\in\Hol_h(R_L)}
       \psi\left(-\sum_{i<j}H_{ij}c_{ij}\right)G_L(H)^M.
\end{equation}
The normalization uses one coordinate per unordered pair; the
factor $1/2$ in $x^tHx/2$ ensures the same convention.

\begin{proposition}\label{prop:count}
There is a constant $K_h$ depending only on $h$ with the following
properties.
\begin{enumerate}
\item If $M>F(h)$ and
$\epsilon=\min_{2\le r\le h}(Mr/2-\delta_h(r))$, then
\begin{equation}\label{eq:strictcount}
 \left|\frac{|\Phi_{M,h}^{-1}(c)(R_L)|}{q^{L(N-e)}}-1\right|
       \le K_hq^{-\epsilon}
\end{equation}
for every $L\ge1$ and every target $c$.
\item For arbitrary $M\ge1$, put
$b=b_h(M/2)$ and $\gamma=e-\max(0,b)$. Then
\begin{equation}\label{eq:generalcount}
 |\Phi_{M,h}^{-1}(c)(R_L)|
   \le K_h(L+1)^{2h-3}q^{L(N-\gamma)}.
\end{equation}
\end{enumerate}
\end{proposition}
\begin{proof}
The term $H=0$ in \eqref{eq:Fourier} is one. Apply the triangle
inequality, Lemma \ref{lem:gauss}, and Lemma \ref{lem:hollow}.
In the strict case $b=-\epsilon$, with $\epsilon\ge1/2$ since
$M,h,r$ are integers. Thus the nonzero mass is at most
\[
 B_h\sum_{\ell=1}^{L}(\ell+1)^{2h-4}q^{-\epsilon\ell}
 \le B_hq^{-\epsilon}
       \sum_{\ell\ge1}(\ell+1)^{2h-4}3^{-(\ell-1)/2}.
\]
The last series is finite and depends only on $h$.
Without strictness, bound every exponential factor by
$q^{L\max(0,b)}$ and sum the polynomial factors.
\end{proof}

\subsection{Quantitative densities over local fields}
Let $\mathcal O$ now be the valuation ring of a nonarchimedean local
field with residue cardinality $q$ odd. Give each $\mathcal O^a$ its
Haar probability measure. For a graph $G$ with $g$ edges, define the
locally constant function on $\mathcal O^g$
\[
 f_{G,L}(c)=q^{-L(N-g)}
   \#\{V\in R_L^{M\times h}:\Phi_{M,G}(V)=c\bmod\pi^L\}.
\]
It is the density of the pushforward after averaging on residue balls
of radius $q^{-L}$.

\begin{proposition}\label{prop:density}
Suppose $M>F(h)$, and put
$\epsilon=\min_{2\le r\le h}(Mr/2-\delta_h(r))$ and $d_h=2h-4$.
There is a constant $A_h$ depending only on $h$ such that, for every
such $\mathcal O$ and every graph $G$, the pushforward of Haar probability
under $\Phi_{M,G}$ has a continuous density $f_G$, and
\begin{align}
 \|f_G-1\|_\infty&\le A_hq^{-\epsilon},\label{eq:densityuniform}\\
 \|f_G-f_{G,L}\|_\infty
   &\le A_h(L+2)^{d_h}q^{-\epsilon(L+1)},\qquad L\ge0,
       \label{eq:densitytail}
\end{align}
where $f_{G,0}=1$. In particular, if $c-c'\in\pi^L\mathcal O^g$, then
\[
 |f_G(c)-f_G(c')|
 \le 2A_h(L+2)^{d_h}q^{-\epsilon(L+1)}.
\]
These statements hold in both equal and mixed characteristic.
\end{proposition}
\begin{proof}
First take $G=K_h$ and write $\mu$ for the pushforward measure. Characters of the compact additive group
$\mathcal O^e$ have finite conductor. Those of exact conductor
$\ell\ge1$ are indexed by primitive hollow matrices over $R_\ell$,
after choosing a generating additive character of $R_\ell$.
The Fourier coefficient of the pushforward at such a character is
$G_\ell(H)^M$: the $M$ rows are independent and Haar-uniform.
Lemmas~\ref{lem:gauss} and \ref{lem:hollow} therefore give
\[
 \sum_{\chi\text{ of conductor }\ell}
       |\widehat\mu(\chi)|
   \le U_h(\ell,M/2)
   \le B_h(\ell+1)^{d_h}q^{-\epsilon\ell}.
\]
Since $\epsilon\ge1/2$, these bounds are summable. Fourier inversion
on each finite quotient identifies its partial sum through conductor
$L$ with $f_{K_h,L}$. The series converges uniformly to a continuous
function $f$. For each residue ball $B$, the integrals
$\int_B f_{K_h,L}$ equal the pushforward measure of $B$ once $L$ is
at least its level. Uniform convergence gives the same identity for
$f$, so $f$ is the density of the pushforward.

For the tail, write $\ell=L+1+j$ and use
$(L+j+2)^{d_h}\le(L+2)^{d_h}(j+1)^{d_h}$. Then
\[
 \sum_{\ell>L}(\ell+1)^{d_h}q^{-\epsilon\ell}
 \le (L+2)^{d_h}q^{-\epsilon(L+1)}
       \sum_{j\ge0}(j+1)^{d_h}3^{-j/2}.
\]
This proves \eqref{eq:densitytail}; summing from $\ell=1$ gives
\eqref{eq:densityuniform}, after increasing $A_h$ if needed.
The finite-level function is constant on each residue ball, which
gives the stated modulus. For a general graph, integrate $f$ over the
unused $e-g$ target coordinates. The same averaging sends
$f_{K_h,L}$ to $f_{G,L}$ and cannot increase a supremum norm. It
preserves continuity because the target is compact. This proves all
claims, including the case of a graph without edges.
\end{proof}

\begin{remark}
For $h=2$ the exponent $\epsilon=M-1$ in the modulus cannot be
increased. Indeed, for $M\ge2$ the map is the dot product of two
vectors in $\mathcal O^M$. The first vector has minimum coordinate
valuation $j$ with probability $(1-q^{-M})q^{-Mj}$. Conditional on
this event, its dot product with a Haar-uniform second vector is
uniform on $\pi^j\mathcal O$. Consequently, with $\nu(0)=\infty$,
\[
 f(c)=(1-q^{-M})\sum_{j=0}^{\nu(c)}q^{-(M-1)j}.
\]
In particular, $f(0)-f(\pi^L)$ is a positive constant times
$q^{-(M-1)(L+1)}$. For general $h$, no optimality of the modulus
is asserted. Its continuity is consistent with the general FRS
theory \cite{AAfrs}; the displayed conductor-dependent bound and
its validity at every odd residue characteristic follow directly
from the hollow-matrix estimate.
\end{remark}

\section{All coefficient fibers}
We now prove Theorem \ref{thm:coeff}. Fix $L$, put $R=\Z[1/2]$
and $d=L(N-e)$, and first consider the complete-graph map
\[
 f:\mathcal X=\A_R^{LN}\longrightarrow\mathcal Y=\A_R^{Le}.
\]
The argument first determines the geometric fibers at closed points,
then transfers their properties to all points, proves flatness, and
finally removes the unwanted edge coordinates.

\begin{proof}[Proof of Theorem \ref{thm:coeff}]
For a finite field $\F_q$ of odd characteristic, use
$R_L=\F_q[t]/(t^L)$ in Proposition \ref{prop:count}.
Fix a coefficient target $c$ defined over $\F_q$ and apply
\eqref{eq:strictcount} over every finite extension. If $\mathcal F=f^{-1}(c)$,
the resulting estimate is
\begin{equation}\label{eq:extensioncount}
 \left|\frac{\#\mathcal F(\F_{q^a})}{q^{ad}}-1\right|
       \le K_hq^{-a\epsilon},\qquad a\ge1,
\end{equation}
with $\epsilon>0$ independent of $a$. Choose a finite extension
defining all geometric irreducible components of the reduced support
of $\mathcal F$, and let $a$ tend to infinity through multiples of that degree.
Apply \cite[Theorem 1]{LW} to the projective closures of its
positive-dimensional components. The boundaries, pairwise intersections,
and components of smaller dimension contribute lower-order terms by
\cite[Lemma 1]{LW}; a zero-dimensional reduced scheme instead has
an eventually constant point count. Thus, if $D=\dim \mathcal F$ and $b_D$
is the number of geometric components of dimension $D$, then
$\#\mathcal F(\F_{q^a})/q^{aD}\to b_D$. Equation~\eqref{eq:extensioncount}
forces $D=d$ and $b_D=1$: the normalized count would otherwise tend
to zero, to infinity, or to a positive integer different from one.

The fiber is defined by $Le$ equations in $\A^{LN}$.
Every irreducible component has dimension at least $d$, by
Krull's height theorem. It follows that there is just one
component. At every point of the fiber, the defining ideal has
height $Le$ and $Le$ generators in a regular local ring. These
generators form a regular sequence there. The fiber is therefore
a local complete intersection and is Cohen--Macaulay.

We check reducedness separately, since point counts cannot detect
nilpotents. Taking $r=h$ in \eqref{eq:F} gives $F(h)\ge h-1$,
so $M>F(h)$ implies $M\ge h$. Over the algebraic closure, complete
the prescribed constant off-diagonal target to an invertible
symmetric $h\times h$ matrix by choosing its diagonal. This is
possible since the determinant contains the product of the
diagonal variables with coefficient one. Realize this matrix
as $V_0^tV_0$ with $V_0$ of full column rank.

For a symmetric matrix $Q$, choose $W$ with $V_0^tW=Q/2$;
this is possible because $V_0^t$ is surjective. Then
$V_0^tW+W^tV_0=Q$, so the differential, including its off-diagonal
part, is surjective. At order $a\ge1$ the Gram coefficient is
\[
 V_0^tV_a+V_a^tV_0+
       \sum_{b=1}^{a-1}V_b^tV_{a-b}.
\]
The sum is already known when $V_a$ is chosen. Thus the off-diagonal
equations can be solved successively. At the resulting point, the
coefficient Jacobian is block lower triangular, with the same
surjective off-diagonal differential on each diagonal block.
It has rank $Le$, so this is a smooth point of the fiber.
The unique component is consequently generically reduced. The
Cohen--Macaulay coordinate ring has no embedded associated primes;
generic reducedness therefore makes it reduced. This proves
geometric integrality for every target over a finite field.

\begin{samepage}
The map $f$ is affine and of finite presentation, and $\mathcal Y$
is Noetherian, quasi-compact, and Jacobson. Consider the subset
\[
 U=\{y\in\mathcal Y: \mathcal X_{\bar y}
       \text{ is reduced, has one irreducible component,
       and has dimension }d\}.
\]
Here $\bar y$ denotes a geometric point over $y$.
By \cite[Tags 0579, 055B, 05F9]{Stacks}, each condition is locally
constructible. Since $\mathcal Y$ is affine, it is quasi-compact and
quasi-separated; \cite[Tag 054E]{Stacks} therefore makes $U$ constructible.\par
\end{samepage}

These constructibility results do not assume flatness.
Every closed point of $\mathcal Y$ has a finite
residue field of odd characteristic, so the preceding argument
puts every closed point in $U$.

For completeness, a nonempty constructible subset of a Noetherian
Jacobson scheme contains a nonempty locally closed subset, hence a
nonempty open subset of some closed subscheme. Closed points are
dense in that closed subscheme and remain closed in the ambient
scheme. Such a subset therefore contains a closed point.
Apply this observation to the constructible complement
$\mathcal Y\setminus U$: it must be empty. We obtain integral
geometric fibers of dimension $d$ at all points, including generic
points and targets with transcendental coordinates. The height
argument above applies to each such fiber and gives the
complete-intersection assertion. All these conclusions concern
geometric fibers and persist under extension of their residue fields.

We verify the local dimension hypothesis for miracle flatness,
including points in positive characteristic. Let $x$ lie over $y$,
and let $\mathfrak p$ be their common image in $\Spec R$. Put
\[
 \eta=\dim\mathcal O_{\Spec R,\mathfrak p}\in\{0,1\},\quad
 b=\operatorname{trdeg}_{\kappa(\mathfrak p)}\kappa(y),\quad
 t=\operatorname{trdeg}_{\kappa(y)}\kappa(x).
\]
The dimension formulas for polynomial rings over $\mathbb Q$ when $\eta=0$
and over the discrete valuation ring $R_{(\mathfrak p)}$ when $\eta=1$ give
\begin{align*}
 \dim\mathcal O_{\mathcal X,x}&=LN+\eta-b-t,\\
 \dim\mathcal O_{\mathcal Y,y}&=Le+\eta-b,\\
 \dim\mathcal O_{\mathcal X_y,x}&=d-t.
\end{align*}
The last equality uses the integral $d$-dimensional fiber over
$\kappa(y)$. Thus the first local dimension is the sum of the
other two. The target local ring is regular and the source local
ring is regular, hence Cohen--Macaulay. Miracle flatness
\cite[Tag 00R4]{Stacks} now gives flatness of $f$.
Its fibers are nonempty, so $f$ is faithfully flat. After any field
base change, flatness pulls the coordinate regular sequence of a
target point back to the displayed coefficient equations. Thus
each fiber is a complete intersection defined by that sequence.

For a graph $G$, put $g=|E(G)|$ and $m=L(e-g)$. Compose $f$
with the coordinate projection $\mathcal Y\to\A_R^{Lg}$.
Both maps are faithfully flat, so their composition is faithfully
flat. To identify a geometric fiber, take any algebraically closed
field $k$ of characteristic different from two and a target
$c_G\in k^{Lg}$. The complete targets extending $c_G$ form
\[
 \mathcal B=\Spec A\simeq\A_k^m,\qquad A=k[z_1,\ldots,z_m].
\]
The graph fiber $Y_G=\Spec B$ is precisely
$\mathcal X\times_{\mathcal Y}\mathcal B$. Consequently $B$ is faithfully
flat over $A$, and every geometric fiber of $Y_G\to \mathcal B$ is integral
of dimension $d$, by the result for $f$.

Let $K=\operatorname{Frac}(A)$. For each nonzero $a\in A$, tensoring
the injection $A\xrightarrow{\,a\,}A$ with the flat $A$-module $B$
shows that multiplication by $a$ is injective on $B$. Localization
therefore embeds $B$ into $B\otimes_A K$. This generic-fiber ring
is a nonzero domain, since it is geometrically integral; hence
$B$ itself is a domain. Moreover,
\[
 \dim B=\operatorname{trdeg}_k\operatorname{Frac}(B)
       =m+\operatorname{trdeg}_K\operatorname{Frac}(B\otimes_AK)
       =m+d=L(N-g).
\]
Here faithful flatness also ensures that $A$ embeds in $B$.
The graph map over $k$ is flat, so the coordinate regular sequence
of $c_G$ pulls back to its $Lg$ displayed coefficient equations
in $k[V_0,\ldots,V_{L-1}]$. They form a regular sequence defining
$Y_G$. This proves geometric integrality,
the complete-intersection property, and the dimension for every
graph fiber. The argument includes $g=0$ and $g=e$, with a
polynomial ring in zero variables interpreted as $k$.
\end{proof}

\begin{corollary}\label{cor:positive}
If $M>F(h)$, every ordinary graph-Gram fiber in characteristic
zero has rational singularities.
\end{corollary}
\begin{proof}
Every jet scheme of the ordinary fiber over $c$ is the coefficient
fiber over $(c,0,\ldots,0)$. Theorem \ref{thm:coeff} makes all
these schemes integral. The ordinary fiber is an integral
complete intersection. Musta\c{t}\u{a}'s theorem
\cite[Theorem 0.1]{Mjets} therefore gives rational singularities.
Since a local complete intersection is Gorenstein, these are equivalently
canonical singularities \cite[discussion following Theorem 0.1]{Mjets}.
\end{proof}

\section{The exact ideal threshold}
\begin{lemma}[The threshold of a homogeneous ideal]\label{lem:cone}
Let $J\subset\C[x_1,\ldots,x_n]$ be a nonzero proper homogeneous
ideal. Then $\lct_{\A^n}(J)=\lct_0(J)$.
\end{lemma}
\begin{proof}
Choose a proper log resolution $\pi:\widetilde Y\to\A^n$ of $J$
that is an isomorphism over $\A^n\setminus V(J)$, and write
$J\mathcal O_{\widetilde Y}=\mathcal O_{\widetilde Y}(-\sum_i a_iE_i)$
and $K_{\widetilde Y/\A^n}=\sum_i k_iE_i$. The resolution formula
\cite[Proposition 1.5]{Mpairs} gives
\[
 \lct_x(J)=\min_{i:\,x\in\pi(E_i),\ a_i>0}\frac{k_i+1}{a_i},
\]
with the minimum of an empty set equal to infinity. Thus, for
each real $c\ge0$, the locus
\[
 B_c=\{x:\lct_x(J)\le c\}
     =\bigcup_{i:\,a_i>0,\ (k_i+1)/a_i\le c}\pi(E_i)
\]
is closed: it is a finite union of images of closed sets under
the proper map $\pi$. Homogeneity means that every dilation
$x\mapsto\lambda x$, $\lambda\in\C^*$, preserves $J$.
The local threshold is invariant under these isomorphisms, so
$B_c$ is invariant under dilation. If $B_c$ is nonempty, it
contains a closed point $x$, and closedness applied to
$\lambda\mapsto\lambda x$ on $\A^1$ implies $0\in B_c$.

The global threshold $c_*=\min_{a_i>0}(k_i+1)/a_i$ is finite
and attained on this resolution, so $B_{c_*}$ is nonempty.
It follows that $\lct_0(J)\le c_*$. The reverse inequality
follows because the local minimum uses a subset of the divisors
in the global minimum. This proves the assertion without any
reducedness or complete-intersection hypothesis on $J$.
\end{proof}

\begin{proof}[Proof of Theorem \ref{thm:lct}]
Set $\gamma=e-\max(0,b_h(M/2))$ as in
Proposition \ref{prop:count}. For each fixed $L$, apply
\eqref{eq:generalcount} to the zero target over every finite
extension of a finite field of odd characteristic. Lang--Weil
gives
\begin{equation}\label{eq:jetupper}
 \dim (X_{M,h})_{L-1}\le L(N-\gamma).
\end{equation}
Here the subscript is jet order. For each fixed $L$, take the
zero-target coefficient scheme over $\Z[1/2]$. Its fiber
dimensions lie between $0$ and $LN$. Each dimension level set is
locally constructible by \cite[Tag 05F9]{Stacks}, hence constructible
by \cite[Tag 054E]{Stacks} because the target is affine. The locus
where the dimension exceeds $L(N-\gamma)$ is a finite union of
such level sets, hence constructible. It contains no closed
point, by the finite-field bound. The Jacobson argument used
above makes it empty. In particular the bound holds over $\mathbb Q$
and, by invariance of dimension under field extension, over~$\C$. This argument is applied separately for every $L$; it
does not require an infinite intersection of constructible sets.

The jet formula for an arbitrary ideal in a smooth ambient
space \cite[Corollary 0.2]{Mpairs} is
\begin{equation}\label{eq:jetlct}
 \lct(I_{M,h})=N-
       \sup_{L\ge1}\frac{\dim(X_{M,h})_{L-1}}{L}.
\end{equation}
No complete-intersection or reducedness hypothesis is imposed
on $I_{M,h}$ in this formula.
Thus \eqref{eq:jetupper} gives $\lct(I_{M,h})\ge\gamma$.
The identity
\[
 e-\delta_h(r)=\binom{h-r+1}{2}
\]
shows that $\gamma$ is the right side of \eqref{eq:lct}.

For the reverse inequality, the zeroth jet in
\eqref{eq:jetlct} gives $\lct(I_{M,h})\le\operatorname{height}
I_{M,h}\le e$. We need only consider the case in which the
inner minimum in \eqref{eq:lct} is less than $e$.
Put $u=\lfloor M/2\rfloor$, the maximal isotropic dimension. To see that $u\le h-2$, observe that
\[
 e-\left(\binom{k+1}{2}+\frac{M(h-k)}2\right)
   =\delta_h(h-k)-\frac{M(h-k)}2.
\]
An inner value below $e$ therefore implies $M<F(h)$. For
$2\le r\le h$, one has $r+2h/r\ge r+2\ge4$, so $F(h)\le2h-3$.
It follows that $u\le h-2$.
The quadratic expression in $k$ is minimized at $k=u$
(also at $k=u-1$ when $M=2u$), and its value is
\begin{equation}\label{eq:tau}
 \tau=
 \begin{cases}
 uh-u(u-1)/2,&M=2u,\\
 ((2u+1)h-u^2)/2,&M=2u+1.
 \end{cases}
\end{equation}
Indeed, the successive difference is $k+1-M/2$. It is negative
for $k<u$ and positive for $k\ge u$ when $M$ is odd; when $M$ is
even, it vanishes at $k=u-1$, giving the two adjacent minimizers.

Let $S$ be the locally closed stratum of matrices whose columns
span a maximal totally isotropic $u$-plane. For even $M$ choose
either component of $\Gr(u,M)$; for $u=0$ take $S=\{0\}$.
The open spanning-frame bundle over this orthogonal Grassmannian
has dimension
\begin{equation}\label{eq:stratum}
 \dim S=
 \begin{cases}
 uh+u(u-1)/2,&M=2u,\\
 uh+u(u+1)/2,&M=2u+1.
 \end{cases}
\end{equation}
For even $M$ the first line says $\dim X_{M,h}\ge N-\tau$.
The zeroth jet in \eqref{eq:jetlct} then gives the required
upper bound.

Suppose $M=2u+1$, and write $r=h-u$. A relation among the
Jacobian rows at $V\in S$ is exactly a hollow symmetric $H$
with $VH=0$. If $K=\ker V$ has dimension $r$, the symmetric
matrices with image in $K$ form a space of dimension
$r(r+1)/2$. For example, after choosing a matrix $P$ whose
columns are a basis of $K$, they are precisely $PAP^t$ with
$A$ symmetric; this parametrization is injective. Requiring
the $h$ diagonal entries to vanish imposes at most $h$
linear conditions. Consequently
\[
 \operatorname{corank}\operatorname{Jac}\Phi_{M,h}|_V
       \ge\frac{r(r+1)}2-h=\binom r2-u=:\delta,
 \qquad
 \dim T_VX_{M,h}\ge N-e+\delta .
\]
The assumption $\tau<e$ implies
$\delta-r/2=e-\tau>0$. Since the first jet over $V$ is
its tangent space, the first jet over $S$ has dimension at least
\[
 \dim S+N-e+\delta
       =N+u^2=2(N-\tau).
\]
Taking $L=2$ in \eqref{eq:jetlct} gives
$\lct(I_{M,h})\le\tau$. This completes the equality.

Finally, $I_{M,h}$ is a nonzero proper homogeneous ideal.
Lemma~\ref{lem:cone} therefore identifies its global threshold
with its threshold at the origin.
\end{proof}

\section{Sharp rationality and the log-canonical boundary}\label{sec:sharpness}
\begin{lemma}[Consequences of integrality]\label{lem:integralCI}
If $X_{M,h}$ is integral, then $M\ge h$ and $X_{M,h}$ is a
complete intersection of dimension $N-e$. These conclusions apply,
in particular, if $X_{M,h}$ has rational singularities.
\end{lemma}
\begin{proof}
Put $d_i=\sum_{a=1}^M v_{ai}^2$ in the coordinate ring of $X_{M,h}$.
The off-diagonal Gram entries vanish in this ring, so
\[
 \prod_{i=1}^h d_i=\det(V^tV).
\]
If $M<h$, the right side is identically zero. However, each $d_i$
is nonzero in the coordinate ring: evaluate at the frame whose
$i$th column is the first coordinate vector and whose other columns
are zero. Thus the ring is not a domain, proving $M\ge h$.

For $M\ge h$, the frame $V_0$ consisting of the first $h$ coordinate
vectors belongs to $X_{M,h}$. At $V_0$ the differential of the
$(i,j)$th equation, for $i<j$, is $w_{ij}+w_{ji}$. These $e$
linear forms are independent, so $V_0$ is smooth of dimension
$N-e$. Integrality therefore forces $\dim X_{M,h}=N-e$.
Its ideal therefore has height $e$ and is generated by $e$ elements
in a polynomial ring, which makes it a complete intersection.

Finally, rational singularities imply normality. All irreducible
components of this affine cone contain its vertex. A normal scheme
has a local domain at that vertex, so it has only one component;
normality also implies reducedness. Hence it is integral.
\end{proof}

\begin{proof}[Proof of Theorem \ref{thm:rat}]
Corollary~\ref{cor:positive} gives rational singularities of every
geometric fiber when $M>F(h)$. Theorem~\ref{thm:coeff} at $L=1$
also gives flatness and reduced fibers, so the map is FRS, for the
complete graph and for every subgraph. Conversely, the FRS property
implies rational singularities of the zero fiber. It remains to prove
that rationality of the zero fiber forces $M>F(h)$. Suppose
that $X=X_{M,h}$ has rational singularities and $M\le F(h)$.
Lemma~\ref{lem:integralCI} gives $M\ge h$, integrality, and the
complete-intersection dimension $n=N-e$.

Set $u=\lfloor M/2\rfloor$. The inequality $F(h)\le2h-3$ follows
from $r+2h/r\ge r+2\ge4$ for $2\le r\le h$. Thus
$M\le2h-3$, and taking the integer part of $M/2$ gives $u\le h-2$.
The successive difference $k+1-M/2$ shows that the inner quadratic
minimum in \eqref{eq:lct} is attained at $k=u$. Its value is the
number $\tau$ in \eqref{eq:tau}. Directly from
$e-\delta_h(r)=\binom{h-r+1}{2}$,
\begin{equation}\label{eq:necessitytau}
 \tau=e-\max_{2\le r\le h}\{\delta_h(r)-Mr/2\}\le e,
\end{equation}
where the inequality is equivalent to $M\le F(h)$.
This is an arithmetic identity; it does not require a primeness
or normality threshold.

Let $S$ be the maximal-isotropic spanning-frame stratum from
\eqref{eq:stratum}, choosing either component when $M$ is even.
Its closure lies in the closed subset
\[
 Z=\{V\in X:V^tV=0\}.
\]
This is a proper subset, since the frame $V_0$ in
Lemma~\ref{lem:integralCI} has $V_0^tV_0=I_h$.

If $M=2u$, equations \eqref{eq:tau} and \eqref{eq:stratum} give
\[
 \dim S=uh+\frac{u(u-1)}2=N-\tau\ge N-e=n.
\]
This contradicts $S\subset Z\subsetneq X$, since a proper closed
subset of an integral $n$-dimensional variety has dimension less
than $n$.

Suppose instead that $M=2u+1$, and put $r=h-u$.
The Jacobian-relation calculation used in the proof of
Theorem~\ref{thm:lct} applies at every $V\in S$: symmetric matrices
with image in $\ker V$ form a space of dimension $r(r+1)/2$,
and imposing zero diagonal adds at most $h$ linear conditions.
Writing $\delta=\binom r2-u$, we obtain
$\dim T_VX\ge n+\delta$. Consequently,
\begin{align*}
 \dim\bigl((X)_1|_S\bigr)
 &\ge\dim S+n+\delta\\
 &=N+u^2=2(N-\tau)\ge2n.
\end{align*}
On the other hand, Musta\c{t}\u{a}'s criterion
\cite[Theorem 0.1]{Mjets} makes $(X)_1$ irreducible. Its open part
over the smooth locus of $X$ has dimension $2n$, hence so does
$(X)_1$. For the jet projection $p:(X)_1\to X$, the subset
$p^{-1}(Z)$ is closed and proper, and contains $(X)_1|_S$.
The displayed dimension bound is again a contradiction.
Both parity arguments include equality $M=F(h)$. Therefore
$M>F(h)$, as required. The complete graph gives necessity for
a condition uniform over all graphs and targets.
\end{proof}

\subsection{Comparison with the normal range}
The rationality proof above is independent of the prime-threshold
theorem. We now use \cite[Theorems 1--2]{CS} to identify the normal
complete-intersection range $M\ge\Dprime(h)$ and compare it with
the rationality and log-canonicity ranges.

We include the integer comparison to specify exactly where
normality and rationality differ. Put $T_k=k(k-1)/2$.

\begin{lemma}\label{lem:arithmetic}
For all $h\ge2$, $\Drat(h)$ is either $\Dprime(h)$ or
$\Dprime(h)+1$. The latter occurs exactly when
\begin{equation}\label{eq:exception}
 T_k+2\le h\le\lfloor k^2/2\rfloor
 \quad\text{for some } k\ge4 .
\end{equation}
\end{lemma}
\begin{proof}
The case $h=2$ is immediate. For $k\ge2$, comparison with
consecutive odd squares gives
\begin{align*}
 \left\lfloor\frac{2h+1-\sqrt{8h+1}}2\right\rfloor=h-k
 &\quad\Longleftrightarrow\quad T_k<h\le T_{k+1},\\
 \left\lfloor\frac{2h+1-\sqrt{8h-7}}2\right\rfloor=h-k
 &\quad\Longleftrightarrow\quad T_k+1<h\le T_{k+1}+1.
\end{align*}
Indeed, the respective square comparisons are
$(2k-1)^2<8h+1\le(2k+1)^2$ and
$(2k-1)^2<8h-7\le(2k+1)^2$.
On $T_k+2\le h\le T_{k+1}$ both floors are $h-k$; at
$h=T_{k+1}+1$ they are $h-k-1$ and $h-k$, respectively.
Taking the minimum in \eqref{eq:Dprime} yields
\[
\begin{array}{ll}
 \Dprime(h)=2h-2k+1,&T_k+2\le h\le T_{k+1},\\
 \Dprime(h)=2h-2k,&h=T_{k+1}+1.
\end{array}
\]
These intervals partition all $h\ge3$.
The successive difference of the expressions maximized in
\eqref{eq:F} is $-1+2h/(r(r+1))$.
On the first interval we may take the maximizing index $r=k$,
and obtain
\[
 F(h)-\Dprime(h)=\frac{k^2-2h}{k}\in[-1,1).
\]
Thus $\lfloor F(h)\rfloor+1$ exceeds $\Dprime(h)$ precisely
when $2h\le k^2$. On the one-point interval the maximizer can
be taken as $k+1$, and
$F(h)-\Dprime(h)=-2/(k+1)\in(-1,0)$.
The exceptional interval is nonempty precisely for $k\ge4$.
\end{proof}

\begin{corollary}\label{cor:lc}
In the normal range $M\ge\Dprime(h)$,
\[
 X_{M,h}\text{ is log canonical}\quad\Longleftrightarrow
       \quad M\ge F(h).
\]
The normal log-canonical but nonrational cases are precisely
\[
 h=2a^2,\qquad M=(2a-1)^2,\qquad a\ge2.
\]
\end{corollary}
\begin{proof}
Put $X=X_{M,h}$, a normal complete intersection of dimension $N-e$.
By \cite[Theorem 1.3]{EM}, $X$ is log canonical if and only if
$(X)_{L-1}$ is equidimensional for every $L\ge1$.
Its $Le$ coefficient equations in $\A^{LN}$ make every component
have dimension at least $L(N-e)$, while the principal component,
the closure of the jets over the smooth locus, has exactly that dimension.
Thus equidimensionality is equivalent to
$\dim(X)_{L-1}=L(N-e)$. By \eqref{eq:jetlct}, these equalities
for all $L$ are equivalent to $\lct(I_{M,h})=e$.
Theorem \ref{thm:lct} makes this equivalent to $M\ge F(h)$.
For log canonicity without rationality, Lemma
\ref{lem:arithmetic} forces an exceptional prime size and
$F(h)=M$. In its notation this means $k^2=2h$.
Hence $k=2a$, and the formulas follow. The exceptional
interval excludes $a=1$.
\end{proof}

\section{Jet dimensions and detection by the first jet}
\begin{corollary}\label{cor:evenjets}
Put $c=\lct(I_{M,h})$. For every $L\ge1$ such that $M$ or $L$
is even,
\[
 \dim(X_{M,h})_{L-1}=L(N-c).
\]
For all $M,h$,
\[
 c=N-\max\left\{\dim X_{M,h},\
                    \frac{\dim(X_{M,h})_1}{2}\right\}.
\]
\end{corollary}
\begin{proof}
The upper bound at every length is \eqref{eq:jetupper}.
If $c=e$, Krull's height theorem gives equality for every $L$.
Suppose $c<e$ and put $u=\lfloor M/2\rfloor$.
For $M=2u$, the smooth isotropic stratum $S$ in
\eqref{eq:stratum} has dimension $N-c$. Its length-$L$
jets have dimension $L(N-c)$ and give the lower bound.

For $M=2u+1$ and $L=2a$, take an algebraic chart of the smooth
space of independent ordered isotropic $u$-frames $A$.
It has dimension $(3u^2+u)/2$. The rank-one quotient
$A^\perp/A$ admits, locally, a representative $z$ whose norm
is an invertible regular function; it need not be normalized to $1$.
Put $r=h-u$. For arbitrary length-$L$ jets $A(t)$ in the
chart, arbitrary $u\times r$ matrices $Y(t)$ modulo $t^L$,
and arbitrary rows $w(t)$ of length $r$ modulo $t^a$, set
\[
 V(t)=\left[A(t)\ \middle|\ A(t)Y(t)+t^az(t)w(t)\right]
                          \pmod {t^{2a}}.
\]
All Gram equations vanish: the first $u$ columns are isotropic
and orthogonal to the other columns, and the remaining
inner products are divisible by $t^{2a}$.
The parameters are recovered from $V(t)$ by the first
$u$ columns and the independent coordinates $A(t),z(t)$.
Thus this is a locally closed family of dimension
\[
 L\left(\frac{3u^2+u}{2}+ur\right)+ar
   =L\left(uh+\frac{u(u+1)}2+\frac r2\right)
   =L(N-c).
\]
For $u=0$ the same construction is simply $V(t)=t^aw(t)$.
For an explicit chart, use the split form
$\langle(x,y,\zeta),(x',y',\zeta')\rangle=x^ty'+y^tx'+\zeta\zeta'$.
Writing $P\in\operatorname{GL}_u$, $\Omega ^t=-\Omega $, and $C$ a
$1\times u$ row, take
\[
 A=\begin{pmatrix}P\\(\Omega -C^tC/2)P\\CP\end{pmatrix},
 \qquad z=\begin{pmatrix}0\\-C^t\\1\end{pmatrix}.
\]
Direct multiplication verifies $A^tJA=0$, $A^tJz=0$,
and $z^tJz=1$, where $J$ is the matrix of the split form.
The parameters $P,\Omega ,C$ give precisely the dimension of
the isotropic-frame chart used above.
This proves the even-length statement. The threshold formula
using only orders zero and one follows by choosing $L=1$
in the even-$M$ case and $L=2$ in the odd-$M$ case, with
the case $c=e$ already covered.
\end{proof}

\begin{corollary}
If $M\ge\Dprime(h)$, then $X_{M,h}$ has rational
singularities if and only if its first jet scheme is
irreducible. It has log-canonical singularities if and
only if its first jet scheme is equidimensional.
\end{corollary}
\begin{proof}
Rationality and log canonicity give the stated jet properties
by \cite{Mjets,EM}. If $M<\Drat(h)$ in the normal range,
Lemma~\ref{lem:arithmetic} gives an exceptional size
$M=\Dprime(h)=2h-2k+1$. The isotropic jet family in the
proof of Theorem~\ref{thm:rat} then has dimension at least
$2\dim X_{M,h}+(k^2-2h)$. At excess zero the first jet is reducible; at
positive excess it is also not equidimensional, since its
principal component still has dimension $2\dim X_{M,h}$.
These are all remaining cases by Lemma \ref{lem:arithmetic}.
\end{proof}

\paragraph{An odd-length example.}
The bound can be strict when both $M$ and $L$ are odd. For
$M=1$, $h=4$, and $L=3$, Theorem~\ref{thm:lct} gives $c=2$,
so \eqref{eq:jetupper} bounds $\dim(X_{1,4})_2$ by $3(4-2)=6$.
In fact its dimension is $5$. Write its four coordinates as
$x_i(t)=a_i+b_it+c_it^2$ modulo $t^3$.
If some $a_i\ne0$, then $x_i(t)$ is a unit, and the equations
$x_i(t)x_j(t)=0$ force $x_j(t)=0$ for every $j\ne i$.
These four open strata have dimension $3$.
On the closed locus $a_1=\cdots=a_4=0$, the only remaining
conditions are $b_ib_j=0$ for $i<j$. The $b$-coordinates form
a union of four coordinate axes, while the four $c_i$ are free.
This closed locus has dimension $1+4=5$, proving the claim.

\section{Reproducible checks and further questions}
The proof rests on Lemmas \ref{lem:symmetric}--\ref{lem:gauss},
Fourier inversion, and the cited geometric theorems.
The four Python programs in the ancillary \texttt{anc/} directory
supply bounded checks
that can expose normalization or arithmetic errors:
\begin{itemize}
\item Exact threshold arithmetic was checked for $2\le h\le20{,}000$,
with an independent brute-force maximization for $h\le600$.
The minimization and isotropic-stratum arithmetic in
\eqref{eq:lct} were checked for $23{,}095$ pairs $(h,M)$.
The direct parity obstruction in the rationality proof was compared
with $M\le F(h)$ for $179{,}101$ pairs with $2\le h\le600$,
without using the prime-threshold formula in that comparison.
\item Exhaustive small-ring matrix enumeration includes all
$531{,}441$ hollow $4\times4$ matrices over $\Z/9$,
$355{,}384$ Schur-pivot checks, and selected independent
kernel counts verifying image lengths.
\item Direct vector enumeration and bit-set intersections give
the same counts as exact Fourier inversion for 22 prescribed
fibers in six cases, including both $\Z/9$ and
$\F_3[t]/(t^2)$. Zero, nonzero, and nonunit targets are included.
\item Four additional checks in odd ambient dimensions compare
first-jet counts obtained from literal Jacobian ranks with
signed Fourier inversion. They include $M=3,h=4$, whose
ideal threshold is $11/2$.
\item The standalone program \texttt{check\_valuation\_strata.py}
computes image lengths by enumerating kernels, without Smith
elimination. It checks all $3{,}048$ hollow matrices in twelve small
cases over $\Z/3^\ell$ and $\F_3[t]/(t^\ell)$, and compares
the counts of $280$ pivot fibers with \eqref{eq:hollowexact}.
It also checks the exact common-valuation decomposition,
\eqref{eq:productcount}, and the mixed-valuation example in
Remark~\ref{rem:mixedvaluation}.
\end{itemize}
The review PDF carries all four scripts, their JSON results, the
ancillary README and checksums as document attachments. It also embeds
\texttt{gram\_verification\_code.zip}: save this file from the PDF
reader's Attachments panel and extract it to obtain \texttt{anc/}.
The accompanying review archive
\texttt{sharp\_gram\_singularity\_thresholds\_review.zip} contains
the same PDF, \texttt{main.tex}, and the complete \texttt{anc/} directory.
The arXiv source archive supplies \texttt{main.tex} and \texttt{anc/};
its ordinary \LaTeX\ build produces the typeset pages without the
review PDF's document attachments. After extracting the code ZIP
or either complete archive, run these commands from its root with
Python 3.11 or later:
\begin{verbatim}
python -X utf8 anc/check_gram_fourier_arithmetic.py
python -X utf8 anc/check_direct_gram_counts.py
python -X utf8 anc/check_odd_gram_first_jets.py
python -X utf8 anc/check_valuation_strata.py
\end{verbatim}
No third-party Python packages are required. Each script writes a
JSON result beside itself, with \texttt{status: PASS} on successful
completion and a SHA-256 hash of its producing script. Keep the four
scripts together: the first-jet checker imports the ring implementation
from the direct-count checker. The supplied JSON files record the
finite cases and exact outputs; \texttt{anc/SHA256SUMS.txt} identifies
the distributed scripts and results. Reruns change elapsed times,
so numerical payloads should be compared after excluding the
\texttt{seconds} field. The ancillary README gives the scope of
each check. These computations supplement the symbolic proofs;
they do not establish the geometric transfer or the cone lemma.

The positive-characteristic question can be stated more precisely using
the prior results. Over an algebraically closed field of odd characteristic,
$F$-rationality of the zero-fiber ring requires $M\ge\Drat(h)$.
Indeed, $F$-rationality implies normality. All components of this
homogeneous scheme meet the origin, so normality forces integrality and
hence $M\ge\Dprime(h)$ by \cite[Theorem 2]{CS}. If
$\Drat(h)=\Dprime(h)+1$, Lemma~\ref{lem:arithmetic} places $h$ in exactly
the exceptional family for which \cite[Theorem 1.3]{Xie} excludes
$F$-rationality at $M=\Dprime(h)$ in every odd characteristic.
On the other hand, \cite[Theorem 3.2]{TV} supplies strong $F$-regularity
when $M\ge3h-4$. Thus the remaining question is whether
$M\ge\Drat(h)$ suffices, uniformly in odd characteristic, even for the
zero fiber. Theorem~\ref{thm:coeff} gives integral coefficient fibers in
that range, but the point counts and the local-field density bounds do
not establish the required Frobenius criteria.

A second question is a uniform formula for jet dimensions at odd
lengths in odd ambient dimension outside the log-canonical range.

\paragraph{Acknowledgment.}
The author thanks OpenAI's GPT-6 for assistance with exploratory
calculations, proof development, code checks, and manuscript drafting.
The author takes responsibility for the mathematical claims and the final
manuscript.

\end{document}